\documentclass[11pt]{amsart}

\usepackage[T1]{fontenc}
\usepackage[utf8]{inputenc}
\usepackage{amsmath,amssymb,amsfonts,mathtools,url}
\usepackage{geometry}
\usepackage{booktabs}
\usepackage{enumitem}
\usepackage{hyperref}
\hypersetup{hidelinks}
\usepackage[nameinlink,capitalize]{cleveref}
\newcommand{\PP}{\mathbb P}
\newcommand{\CC}{\mathbb C}
\newcommand{\QQ}{\mathbb Q}
\newcommand{\NN}{\mathbb N}
\newcommand{\cA}{\mathcal A}
\newcommand{\cB}{\mathcal B}

\newcommand{\cL}{\mathcal L}
\newcommand{\cO}{\mathcal O}
\newcommand{\AR}{\operatorname{AR}}
\newcommand{\mdr}{\operatorname{mdr}}
\newcommand{\rank}{\operatorname{rank}}

\theoremstyle{plain}
\newtheorem{theorem}{Theorem}[section]
\newtheorem{proposition}[theorem]{Proposition}
\newtheorem{lemma}[theorem]{Lemma}
\newtheorem{corollary}[theorem]{Corollary}
\newtheorem{conjecture}[theorem]{Conjecture}
\theoremstyle{definition}

\newtheorem{remark}[theorem]{Remark}

\numberwithin{equation}{section}

\title[A nine-line counterexample]{A nine-line counterexample to a conjecture on the minimal degree of Jacobian relations}

\author{Alexandru Dimca}
\address{Universit\'e C\^ote d'Azur, CNRS, LJAD, France and Simion Stoilow Institute of Mathematics, P.O. Box 1-764 \\ RO-014700 Bucharest, Romania}
\email{Alexandru.DIMCA@univ-cotedazur.fr}

\author{Piotr Pokora}
\thanks{Corresponding author: Piotr Pokora \\ Email of the corresponding author: piotrpkr@gmail.com}
\address{Department of Mathematics, UKEN Krakow, Podchor\c a\.zych 2, PL-30-084 Krak\'ow, Poland}
\email{piotrpkr@gmail.com, piotr.pokora@uken.krakow.pl}

\date{\today}

\begin{document}
\begin{abstract}
We construct two arrangements of nine lines in the complex projective plane with isomorphic intersection lattices, of weak combinatorics \((n_2,n_3,n_4)=(9,7,1)\), whose defining equations \(f\) and \(g\) satisfy
\[
\mdr(f)=4,\qquad \mdr(g)=5.
\]
Since \(4<9/2\), this disproves a conjecture of the first author which would have implied Terao's Conjecture on the combinatorial nature of the freeness of line arrangements.  Equivalently, the rank two vector bundle of logarithmic vector fields is strictly semistable for the first arrangement and stable for the second one.  Hence the stability of this bundle is not determined by the intersection lattice, which gives a counterexample to a conjecture of Abe.  The two arrangements also have different minimal free resolutions of their Milnor algebras.  All computations are exact, and a short \texttt{SINGULAR} script performing them is included.
\end{abstract}

\maketitle

\section{Introduction}
Let \(S=\CC[x,y,z]\), and let \(C:f=0\) be a reduced plane curve of degree \(d\), with Jacobian ideal \(J_f=(f_x,f_y,f_z)\subset S\) and Milnor algebra \(M(f)=S/J_f\).  The graded module of Jacobian syzygies is
\[
\AR(f)=\{(a,b,c)\in S^3\;:\; a f_x+b f_y+c f_z=0\}.
\]
Equivalently, \(\AR(f)\) is the module \(D_0(f)\) of logarithmic derivations annihilating \(f\), after the usual identification \((a,b,c)\leftrightarrow a\partial_x+b\partial_y+c\partial_z\), see \cite[Section 8.1]{DimcaBook}. The first numerical invariant of this module is the minimal degree of a Jacobian relation
\[
\mdr(f)=\min\{q\in\NN:\AR(f)_q\neq 0\}.
\]
For a line arrangement \(\cA:f=0\) we also write \(\mdr(\cA)=\mdr(f)\).

The line arrangement \(\cA\) is free if the graded \(S\)-module \(D_0(f)\) is free. The celebrated Terao's Conjecture claims that if \(\cA\) and \(\cB\) are line arrangements in \(\PP^2\) such that \(\cA\) is free and \(\cA\) and \(\cB\) have isomorphic intersection lattices, then \(\cB\) is also free. This conjecture has attracted a lot of interest, and it is known to hold for arrangements involving at most 14 lines, see \cite{BK, DIM13}.

If \(\cA:f=0\) is an arrangement of \(d\) lines, then the intersection lattice determines the numbers \(n_k\) of points of multiplicity \(k\), hence also the total Tjurina number
\[
\tau(\cA)=\sum_{k\geq 2}n_k(k-1)^2.
\]
It does not, however, determine all Jacobian-syzygy data in general.  The classical examples of Ziegler \cite{Ziegler} and Yuzvinsky \cite{Yuzvinsky} give arrangements with the same intersection lattice but different Jacobian-syzygy behaviour. Geometrically speaking, the construction is related to a conic condition in Pascal-type configurations \cite{DimcaSticlaruPascal}.  The range \(\mdr(f)<d/2\) has been expected to be more rigid.  In particular, the first author formulated the following strengthening of Terao's Conjecture.

\begin{conjecture}[{\cite[Conjecture~3.5]{DimcaOpen}}]
\label{mainc}
\emph{Let \(\cA:f=0\) be a line arrangement in \(\PP^2\), with \(d=\deg f\).  If \(\mdr(f)<d/2\), then \(\mdr(f)\) is combinatorially determined.}
\end{conjecture}

This conjecture would imply Terao's Conjecture, since the arrangement \(\cA\) is free if and only if \(r=\mdr(f)\leq (d-1)/2\) and
\[
\tau(\cA)=(d-1)^2-r(d-r-1),
\]
see \cite{duPCTC, Dmax}.  The following weaker conjecture, which still implies Terao's Conjecture, is open, see \cite[Conjecture~3.7]{DimcaOpen}.

\begin{conjecture}
\label{C2}
\emph{Let \(\cA:f=0\) be a line arrangement in \(\PP^2\), with \(d=\deg f\).  If \(\mdr(f)< (d-2)/2\), then \(\mdr(f)\) is combinatorially determined.}
\end{conjecture}

Conjecture~\ref{C2} is equivalent to claiming that the freeness defect \(\nu(\cA)\) recalled in \cref{rem:nu} is determined by the combinatorics.  In more geometrical terms, Conjecture~\ref{C2} is equivalent to claiming that the generic splitting type of the rank two vector bundle on \(\PP^2\) associated to the graded \(S\)-module \(\AR(f)\) is determined by the combinatorics.  This in turn is an open question asked in \cite[Question 7.12]{Cook+}, see \cite[Conjecture 3.9]{DimcaOpen} for details.

Conjecture~\ref{mainc} also has a natural geometric meaning.  The sheaf \(E_C=\widetilde{D_0(f)}\) is a rank two vector bundle on \(\PP^2\). For a line arrangement it is, up to a twist, the bundle of logarithmic vector fields along \(\cA\).  As recalled in \cref{lem:stability} below, \(E_C\) is stable if and only if \(\mdr(f)\geq d/2\).  Thus the hypothesis of Conjecture~\ref{mainc} says precisely that \(E_C\) is \emph{not stable}, and Conjecture~\ref{mainc} implies that the stability of \(E_\cA\) is determined by the intersection lattice of \(\cA\).  This last statement, together with the analogous statement for semistability, was conjectured by Abe \cite[Conjecture~2.9]{Abe09}, who calls a \(3\)-arrangement (semi)stable if this bundle is (semi)stable, and who proved the conjecture for the family of \(B_2\)-type arrangements.

The purpose of this note is to disprove Conjecture~\ref{mainc}, and with it the stability part of Abe's conjecture, by an explicit example.  Let \(\cA:f=0\) and \(\cB:g=0\) be the arrangements of nine lines defined over \(\QQ\) by
\begin{equation}\label{eq:f-def}
f=(y-z)(y+2z)(2y+z)(x-2y-z)(x-y-2z)(x-y+z)(x+y-z)(x+y+2z)(2x-y-2z),
\end{equation}
\begin{equation}\label{eq:g-def}
g=x(x-y)(x+y)(x-y-z)(y+z)z(x-z)(x+y+2z)(x-2y-z).
\end{equation}

\begin{theorem}[Main Theorem]\label{thm:main}
The arrangements \(\cA\) and \(\cB\) have the following properties.
\begin{enumerate}[label=\textup{(\arabic*)}]
\item They have isomorphic intersection lattices, with common weak combinatorics
\[
(n_2,n_3,n_4)=(9,7,1).
\]
In particular, \(\tau(\cA)=\tau(\cB)=46\).
\item \(\mdr(f)=4\) and \(\mdr(g)=5\).
\item The bundle \(E_\cA\) is strictly semistable, i.e., semistable but not stable, while \(E_\cB\) is stable.
\item The modules \(D_0(f)\) and \(D_0(g)\) have minimal graded free resolutions
\begin{align}
&0\longrightarrow S(-7)\longrightarrow S(-4)\oplus S(-5)\oplus S(-6)\longrightarrow D_0(f)\longrightarrow 0,\label{eq:D0-f-resolution}\\
&0\longrightarrow S(-6)^2\longrightarrow S(-5)^4\longrightarrow D_0(g)\longrightarrow 0.\label{eq:D0-g-resolution}
\end{align}
Hence, in the hierarchy of \cite{ADP}, \(\cA\) has type \(1\), i.e., it is plus-one generated \cite{AbePOG}, while \(\cB\) has type \(2B\).
\end{enumerate}
\end{theorem}

\begin{corollary}\label{cor:main}
\begin{enumerate}[label=\textup{(\alph*)}]
\item Conjecture~\ref{mainc} is false.
\item The stability of the bundle \(E_\cA\) of a line arrangement \(\cA\subset\PP^2\) is not determined by the intersection lattice of \(\cA\).  In particular, the stability part of \cite[Conjecture~2.9]{Abe09} does not hold.
\end{enumerate}
\end{corollary}

\begin{proof}
By \cref{thm:main}(1), \(\cA\) and \(\cB\) have isomorphic intersection lattices.  By \cref{thm:main}(2), \(\mdr(\cA)=4<9/2\), so \(\cA\) satisfies the hypothesis of Conjecture~\ref{mainc}, but \(\mdr(\cB)=5\neq\mdr(\cA)\).  This proves (a), and (b) follows from \cref{thm:main}(3).
\end{proof}

We add some comments.  The pair \((\cA,\cB)\) is not a rephrasing of the Ziegler--Yuzvinsky pair.  That pair has only double and triple points, with weak combinatorics \((n_2,n_3)=(18,6)\), and its two values of \(\mdr\) are \(5\) and \(6\), according to whether the six triple points lie on a conic or not \cite{DimcaSticlaruPascal}.  Since both values are at least \(9/2\), both bundles in that pair are stable, so it is compatible with Conjecture~\ref{mainc} and with Abe's conjecture.  Our arrangements have one quadruple point and seven triple points, and \cref{sec:search} explains how this weak combinatorics is singled out among nine-line arrangements. On the other hand, the example does not contradict Conjecture~\ref{C2}, nor the semistability part of Abe's conjecture since both bundles are semistable, and they have the same freeness defect and the same generic splitting type, see \cref{rem:nu}.  In fact, for an odd number of lines the hypothesis of Conjecture~\ref{C2} says exactly that \(E_\cA\) is not semistable (\cref{lem:stability}), so Conjecture~\ref{C2} would imply the semistability part of Abe's conjecture for such arrangements.  To the best of our knowledge, no counterexample to Abe's conjecture had been recorded before.

The proof of \cref{thm:main} is a finite computation over \(\QQ\).  In \cref{sec:proof}, we explain which computations are needed and how they imply the theorem. The \texttt{SINGULAR} \cite{Singular} script in \cref{app:code} performs all necessary computations in the blink of an eye.

\section{Proof of the Main Theorem}
\label{sec:proof}

\subsection{Jacobian syzygies and the associated bundle}
Let \(C:f=0\) be a reduced plane curve of degree \(d\).  For each \(q\geq0\) consider the linear map
\begin{equation}\label{eq:Phi-map}
\Phi_{f,q}:S_q^3\longrightarrow S_{q+d-1},
\qquad
(a,b,c)\longmapsto a f_x+b f_y+c f_z,
\end{equation}
so that \(\AR(f)_q=\ker \Phi_{f,q}\).  If \(\rho\in D_0(f)_q\) is non-zero, then \(x\rho\in D_0(f)_{q+1}\) is non-zero.  Hence \(D_0(f)_q=0\) if and only if \(q<\mdr(f)\), and
\begin{equation}\label{eq:mdr-criterion}
\mdr(f)=q \iff \Phi_{f,q-1} \text{ is injective and } \Phi_{f,q} \text{ is not injective.}
\end{equation}
In the monomial bases, \(\Phi_{f,q}\) is given by a \(\binom{q+d+1}{2}\times 3\binom{q+2}{2}\) matrix whose entries are coefficients of \(f_x,f_y,f_z\). When \(f\) has rational coefficients, its rank is an exact linear-algebra computation over \(\QQ\).

The module \(D_0(f)\) is a second syzygy module of rank two, hence it is reflexive and has projective dimension at most one.  If its minimal homogeneous generators have degrees \(d_1\leq d_2\leq \cdots \leq d_m\), called the exponents of \(C\), then \(d_1=\mdr(f)\) and the minimal resolution of \(D_0(f)\) has the form
\[
0\longrightarrow \bigoplus_{j=1}^{m-2}S(-c_j)\longrightarrow \bigoplus_{i=1}^{m}S(-d_i)\longrightarrow D_0(f)\longrightarrow 0.
\]
Shifting by \(d-1\), one obtains the minimal resolution of the Milnor algebra:
\begin{equation}\label{eq:Milnor-resolution-general}
0\longrightarrow \bigoplus_{j}S(-(d-1+c_j))\longrightarrow \bigoplus_{i}S(-(d-1+d_i))\longrightarrow S(-(d-1))^3\longrightarrow S\longrightarrow M(f)\longrightarrow 0.
\end{equation}
Following \cite{ADP}, the type of \(C\) is \(t(C)=d_1+d_2+1-d\).  Type \(0\) curves are precisely free curves, and type \(1\) curves are plus-one generated.  In type \(2\) there are two basic homological alternatives: type \(2A\), corresponding to a three-syzygy curve, and type \(2B\), corresponding to a four-syzygy curve.

Sheafifying the exact sequence \(0\to D_0(f)\to S^3\to J_f(d-1)\to 0\) gives
\[
0\longrightarrow E_C\longrightarrow \cO_{\PP^2}^3\longrightarrow \mathcal{I}_\Sigma(d-1)\longrightarrow 0,
\]
where \(\Sigma\) is the Jacobian subscheme of \(C\), of degree \(\tau(C)\).  Hence \(E_C=\widetilde{D_0(f)}\) is a rank two vector bundle with
\begin{equation}\label{eq:chern}
c_1(E_C)=1-d,\qquad c_2(E_C)=(d-1)^2-\tau(C),
\end{equation}
and, \(D_0(f)\) being reflexive, \(H^0(\PP^2,E_C(k))=D_0(f)_k\) for all \(k\in\mathbb Z\).  Stability is meant in the sense of Mumford--Takemoto, as in \cite{Abe09}.  The following fact is well known, and we include the short proof for completeness, see also \cite{DS20}.

\begin{lemma}\label{lem:stability}
Let \(C:f=0\) be a reduced plane curve of degree \(d\).  Then \(E_C\) is stable if and only if \(\mdr(f)\geq d/2\), and \(E_C\) is semistable if and only if \(\mdr(f)\geq (d-1)/2\).  Consequently:
\begin{enumerate}[label=\textup{(\roman*)}]
\item the hypothesis \(\mdr(f)<d/2\) of Conjecture~\ref{mainc} holds if and only if \(E_C\) is not stable;
\item if \(d\) is odd, the hypothesis \(\mdr(f)<(d-2)/2\) of Conjecture~\ref{C2} holds if and only if \(E_C\) is not semistable.
\end{enumerate}
\end{lemma}

\begin{proof}
Recall \cite[Chapter~II, Lemma~1.2.5]{OSS} that a rank two vector bundle \(F\) on \(\PP^2\) with \(c_1(F)\in\{0,-1\}\) is stable if and only if \(H^0(F)=0\). If \(c_1(F)=-1\), semistability is equivalent to stability, and if \(c_1(F)=0\), \(F\) is semistable if and only if \(H^0(F(-1))=0\).  If \(d=2k+1\) is odd, then \(c_1(E_C(k))=0\), so \(E_C\) is stable if and only if \(D_0(f)_k=0\), i.e., \(\mdr(f)\geq k+1\), and semistable if and only if \(D_0(f)_{k-1}=0\), i.e., \(\mdr(f)\geq k\).  If \(d=2k\) is even, then \(c_1(E_C(k-1))=-1\), so \(E_C\) is (semi)stable if and only if \(D_0(f)_{k-1}=0\), i.e., \(\mdr(f)\geq k\).  Since \(\mdr(f)\) is an integer, these conditions are the ones stated.  Claim (i) is now clear, and (ii) follows since, for \(d=2k+1\), both \(\mdr(f)<(d-2)/2\) and \(\mdr(f)<(d-1)/2\) mean \(\mdr(f)\leq k-1\).
\end{proof}

\subsection{The lattice}
We label the lines in the order of \eqref{eq:f-def} and \eqref{eq:g-def}:
\[
\begin{array}{lll}
L_1:y-z=0, & L_2:y+2z=0, & L_3:2y+z=0,\\
L_4:x-2y-z=0, & L_5:x-y-2z=0, & L_6:x-y+z=0,\\
L_7:x+y-z=0, & L_8:x+y+2z=0, & L_9:2x-y-2z=0,
\end{array}
\]
\[
\begin{array}{lll}
M_1:x=0, & M_2:x-y=0, & M_3:x+y=0,\\
M_4:x-y-z=0, & M_5:y+z=0, & M_6:z=0,\\
M_7:x-z=0, & M_8:x+y+2z=0, & M_9:x-2y-z=0.
\end{array}
\]

\begin{lemma}\label{lem:lattice}
With this labelling, the points of multiplicity at least three of both arrangements are exactly the following ones:
\begin{center}
\begin{tabular}{c|cccccccc}
\toprule
lines & \(\{1,2,3\}\) & \(\{1,4,5\}\) & \(\{1,6,7\}\) & \(\{2,4,6\}\) & \(\{3,5,7\}\) & \(\{3,6,8\}\) & \(\{4,7,9\}\) & \(\{2,5,8,9\}\)\\
\midrule
point in \(\cA\) & \((1{:}0{:}0)\) & \((3{:}1{:}1)\) & \((0{:}1{:}1)\) & \((3{:}2{:}{-1})\) & \((3{:}{-1}{:}2)\) & \((3{:}1{:}{-2})\) & \((1{:}0{:}1)\) & \((0{:}2{:}{-1})\)\\
point in \(\cB\) & \((0{:}0{:}1)\) & \((0{:}1{:}{-1})\) & \((0{:}1{:}0)\) & \((1{:}1{:}0)\) & \((1{:}{-1}{:}1)\) & \((1{:}{-1}{:}0)\) & \((1{:}0{:}1)\) & \((1{:}1{:}{-1})\)\\
\bottomrule
\end{tabular}
\end{center}
All other intersection points are double points.  Consequently, the map \(L_i\mapsto M_i\) induces an isomorphism of intersection lattices, and the common weak combinatorics is \((n_2,n_3,n_4)=(9,7,1)\), with \(\tau(\cA)=\tau(\cB)=46\).
\end{lemma}

\begin{proof}
One checks directly that each listed point lies on the listed lines.  Represent a line \(ax+by+cz=0\) by the vector \((a,b,c)\).  Since the lines are pairwise distinct, three of them are concurrent if and only if the corresponding \(3\times3\) determinant vanishes.  For both arrangements, exactly \(11\) of the \(\binom93=84\) determinants vanish, namely those of the seven triples in the table and of the four \(3\)-subsets of \(\{2,5,8,9\}\), please consult part (1) of the script in \cref{app:code}.  Hence there are no further multiple points, the lattices are isomorphic via \(L_i\mapsto M_i\), \(n_4=1\), \(n_3=7\), and \(n_2=\binom92-3n_3-6n_4=9\). 
\end{proof}

\subsection{The minimal degree of a Jacobian relation}
Since \(d=9\), the maps \eqref{eq:Phi-map} are \(\Phi_{h,q}:S_q^3\to S_{q+8}\) for \(h\in\{f,g\}\).  Their ranks, computed over \(\QQ\) by part (2) of the script in \cref{app:code}, are as follows.

\begin{table}[ht]
\begin{tabular}{c|c|c|c|c|c}
\toprule
\(q\) & \(\dim S_q^3\) & \(\rank \Phi_{f,q}\) & \(\dim\AR(f)_q\) & \(\rank \Phi_{g,q}\) & \(\dim\AR(g)_q\) \\
\midrule
0 & 3  & 3  & 0 & 3  & 0 \\
1 & 9  & 9  & 0 & 9  & 0 \\
2 & 18 & 18 & 0 & 18 & 0 \\
3 & 30 & 30 & 0 & 30 & 0 \\
4 & 45 & 44 & 1 & 45 & 0 \\
5 & 63 & 59 & 4 & 59 & 4 \\
6 & 84 & 74 & 10 & 74 & 10 \\
7 & 108& 90 & 18 & 90 & 18 \\
\bottomrule
\end{tabular}
\caption{Ranks of the maps \(\Phi_{f,q}\) and \(\Phi_{g,q}\).}
\label{tab:ranks}
\end{table}

\begin{proposition}\label{prop:mdr}
\(\mdr(f)=4\) and \(\mdr(g)=5\).
\end{proposition}

\begin{proof}
By \eqref{eq:mdr-criterion}, only four facts are needed.
\begin{enumerate}[label=\textup{(\alph*)}]
\item \(\rank\Phi_{f,3}=30\): the \(78\times30\) integer matrix of \(\Phi_{f,3}\) has full column rank.  Hence \(\mdr(f)\geq4\).
\item \(D_0(f)_4\neq0\): the triple \(\rho_f=(A,B,C)\) given below satisfies \(Af_x+Bf_y+Cf_z=0\), as one checks by expanding (part (3) of the script).  Hence \(\mdr(f)=4\).
\item \(\rank\Phi_{g,4}=45\): the \(91\times45\) integer matrix of \(\Phi_{g,4}\) has full column rank.  Hence \(\mdr(g)\geq5\).
\item \(D_0(g)_5\neq0\): indeed \(\rank\Phi_{g,5}=59<63\).  This also follows from the general inequality \(\mdr\leq d-m\), where \(m=4\) is the maximal multiplicity of a point of \(\cB\), see \cite{Dimca17}.  Hence \(\mdr(g)=5\). \qedhere
\end{enumerate}
\end{proof}

The explicit degree-four syzygy for \(f\) used in (b) is
\[
\begin{aligned}
A={}&-x(6x^3-17x^2y-10x^2z-14xy^2-20xyz-47xz^2+33y^3-15y^2z+75yz^2+69z^3),\\
B={}&12x^3y-28x^2y^2+10x^2yz+18x^2z^2-4xy^3+20xy^2z-61xyz^2-36xz^3\\
&+12y^4-57y^3z+33y^2z^2+75yz^3+18z^4,\\
C={}&12x^3z+18x^2y^2-x^2yz-17x^2z^2-40xy^2z-34xyz^2-7xz^3\\
&-18y^4+39y^3z+69y^2z^2-3yz^3-6z^4.
\end{aligned}
\]

\subsection{Resolutions, stability, and the end of the proof}
For \(h\in\{f,g\}\), the number of minimal generators of \(D_0(h)\) in degree \(q\) is \(\mu_q(D_0(h))=\dim D_0(h)_q/S_1D_0(h)_{q-1}\).  These numbers are as follows (see also Proposition \ref{prop:D0-resolutions}):
\[
\begin{array}{c|cccccccc}
q &0&1&2&3&4&5&6&7\\
\hline
\mu_q(D_0(f))&0&0&0&0&1&1&1&0\\
\mu_q(D_0(g))&0&0&0&0&0&4&0&0
\end{array}
\]
For \(f\), the syzygy \(\rho_f\) has three independent multiples in degree five, and \(\dim D_0(f)_5=4\), so there is exactly one new generator \(\rho_5\) in degree five.  Since \(\rho_f\) is primitive (otherwise dividing it by a common factor would produce a syzygy of degree less than \(4\)), \(\rho_f\) and \(\rho_5\) satisfy no relation of degree \(6\), so \(\dim S_1D_0(f)_5=6+3=9\) and there is exactly one new generator in degree six.  For \(g\), all four degree-five syzygies are new.

\begin{proposition}\label{prop:D0-resolutions}
The minimal graded free resolutions of \(D_0(f)\) and \(D_0(g)\) are \eqref{eq:D0-f-resolution} and \eqref{eq:D0-g-resolution}.  Equivalently, by \eqref{eq:Milnor-resolution-general}, the Milnor algebras have minimal resolutions
\begin{align*}
&0\to S(-15)\to S(-12)\oplus S(-13)\oplus S(-14)\to S(-8)^3\to S\to M(f)\to 0,\\
&0\to S(-14)^2\to S(-13)^4\to S(-8)^3\to S\to M(g)\to 0.
\end{align*}
\end{proposition}

\begin{proof}
The resolutions of \(M(f)\) and \(M(g)\) are computed by part (5) of the script in \cref{app:code}, and the resolutions of \(D_0(f)\), \(D_0(g)\) follow by shifting.  They agree with \cref{tab:ranks} and with the numbers \(\mu_q\) above: for \(f\), the generators in degrees \(4,5,6\) contribute \(\binom{5}{2}+\binom{4}{2}+\binom{3}{2}=19\) to degree \(7\), while \(\dim D_0(f)_7=18\), which accounts for the relation in degree \(7\); for \(g\), the four generators of degree \(5\) contribute \(4\dim S_1=12\) to degree \(6\), while \(\dim D_0(g)_6=10\), which accounts for the two relations in degree \(6\).
\end{proof}

\begin{proof}[Proof of \cref{thm:main}]
Part (1) is \cref{lem:lattice} and part (2) is \cref{prop:mdr}.

For part (3), note that \(d=9\), so by \eqref{eq:chern} the twists \(E_\cA(4)\) and \(E_\cB(4)\) have \(c_1=0\) and \(c_2=18-32+16=2\).  The syzygy \(\rho_f\) is a non-zero section of \(E_\cA(4)\), so \(\cO_{\PP^2}\to E_\cA(4)\) is a subsheaf of slope \(0=\mu(E_\cA(4))\) and \(E_\cA\) is not stable since \(H^0(E_\cA(3))=D_0(f)_3=0\), \(E_\cA\) is semistable.  Since \(H^0(E_\cB(4))=D_0(g)_4=0\), \(E_\cB\) is stable.  This is Lemma \ref{lem:stability} with \(d=9\): stability means \(\mdr\geq5\) and semistability means \(\mdr\geq4\).

Part (4) is Proposition \ref{prop:D0-resolutions}.  The exponents are \((4,5,6)\) for \(\cA\) and \((5,5,5,5)\) for \(\cB\), so
\[
t(\cA)=4+5+1-9=1,\qquad t(\cB)=5+5+1-9=2.
\]
Thus \(\cA\) is plus-one generated; since \(d_1+d_2=d\) and \(d_3=d_2+1\), it is even minimal plus-one generated in the sense of \cite[Definition 1.1]{mpog}.  Since \(D_0(g)\) has four minimal generators, \(\cB\) is a four-syzygy curve, i.e., of type \(2B\) in the terminology of \cite{ADP}.
\end{proof}

\begin{remark}\label{rem:hilbert}
By \cref{prop:D0-resolutions}, the Hilbert series of the Milnor algebras are
\[
H_{M(f)}(t)=\frac{1-3t^8+t^{12}+t^{13}+t^{14}-t^{15}}{(1-t)^3},
\qquad
H_{M(g)}(t)=\frac{1-3t^8+4t^{13}-2t^{14}}{(1-t)^3}.
\]
For instance \(\dim M(f)_{12}=47\) and \(\dim M(g)_{12}=46\), and both Hilbert functions stabilize to the common Tjurina number \(46\).
\end{remark}

\begin{remark}[Freeness defect and semistability]\label{rem:nu}
By \cite{DimcaOpen}, if \(C:f=0\) has degree \(d\), total Tjurina number \(\tau(C)\), and \(r=\mdr(f)\), then the freeness defect is
\[
\nu(C)=
\begin{cases}
(d-1)^2-r(d-1-r)-\tau(C) & \text{if } r<(d-2)/2,\\[2pt]
\left\lceil \tfrac34(d-1)^2\right\rceil-\tau(C) & \text{if } r\geq (d-2)/2.
\end{cases}
\]
Here \(d=9\), \(\tau=46\) and \(r\in\{4,5\}\), so \(\nu(\cA)=\nu(\cB)=48-46=2\).  Correspondingly, since both \(E_\cA(4)\) and \(E_\cB(4)\) are semistable with \(c_1=0\), the Grauert--M\"ulich theorem \cite[Chapter~II, Section~2]{OSS} shows that both restrict to \(\cO_L\oplus\cO_L\) on a generic line \(L\), that is \(\cA\) and \(\cB\), have the same generic splitting type \((4,4)\).  Thus the example attacks precisely the combinatoriality of \(\mdr\) in the range of Conjecture~\ref{mainc}, i.e., the stability of \(E_\cA\), and it is compatible with Conjecture~\ref{C2} and with the semistability part of Abe's conjecture.  Finally, the pair \((\cA,\cB)\) realizes the only two possibilities for curves with \(\nu=2\) described in \cite[Remark~1.6]{mpog}: \(\cA\) is minimal plus-one generated, whereas \(\cB\) is a maximal Tjurina curve of type \((2r-1,r)\) with \(r=5\), since \(\tau(\cB)=46\) is the maximal value of \(\tau\) for \(d=9\) and \(r=5\) \cite{duPCTC}.
\end{remark}

\section{How the example was found}
\label{sec:search}

We briefly explain how the weak combinatorics \((9,7,1)\) is singled out.  Let \(\cL:h=0\) be an arrangement of nine lines, \(r=\mdr(h)\), and let \(m\) be the maximal multiplicity of an intersection point.  Recall that \(\tau(\cL)=\sum_{k}n_k(k-1)^2\) and \(\sum_{k}\binom{k}{2}n_k=\binom92=36\).

If \(2m>9\), then \(r=9-m\) \cite{Dimca17}, and if \(m=2\), the arrangement is generic and \(r=7\).  Hence two realizations of one lattice with different values of \(\mdr\) can only occur for \(m=3\) or \(m=4\).  For \(m=3\) one has \(\tau=36+n_3\) and \(4\leq r\leq 6\). Observe that the classical Ziegler--Yuzvinsky weak combinatorics \((18,6)\) lies in this branch.  Looking for a genuinely different combinatorial type, we turned to \(m=4\).  Then
\begin{equation}\label{eq:m4}
n_2+3n_3+6n_4=36,\qquad \tau(\cL)=36+T,\quad T=n_3+3n_4,
\end{equation}
and the usual estimates give \(3\leq r\leq5\).  Moreover \(1\leq n_4\leq 3\), since an elementary incidence count shows that four quadruple points require at least ten lines.  By \cite{duPCTC}, the maximal values of \(\tau\) for \(d=9\) are
\[
\begin{array}{c|ccc}
r&3&4&5\\
\hline
\tau(9,r)_{\max}&49&48&46
\end{array}
\]
Hence \(T\leq 13\), and the values of \(r\) not excluded numerically are \(r=3\) for \(T=13\), \(r\in\{3,4\}\) for \(T\in\{11,12\}\), and \(r\in\{3,4,5\}\) for \(T\leq 10\).  Since \(d/2=9/2\), a jump \(4\leftrightarrow5\) crosses the threshold of Conjecture~\ref{mainc}, and equivalently it changes the stability of the bundle (\cref{lem:stability}).  The first level, descending from the extremal case, at which both \(r=4\) and \(r=5\) are numerically possible is \(T=10\), and by \eqref{eq:m4} there are exactly three weak combinatorics with \(m=4\) and \(T=10\):
\[
(n_2,n_3,n_4)\in\{(9,7,1),\ (12,4,2),\ (15,1,3)\}.
\]
The simplest one is \((9,7,1)\), with \(\tau=46=\tau(9,5)_{\max}\). Thus an arrangement of this type with \(r=5\) lies on the maximal Tjurina boundary, while \(r=4\) is not excluded.

\section*{Acknowledgements}
We want to warmly thank the referee for helpful suggestions, in particular for pointing out the consequence of our example for Abe's conjecture on the stability of arrangements.

\section*{Funding}
Alexandru Dimca is partially supported from the project ``Singularities and Applications'' - CF 132/31.07.2023 funded by the European Union - NextGenerationEU - through Romania's National Recovery and Resilience Plan.

Piotr Pokora is supported by the National Science Centre (Poland) Sonata Bis Grant
\[\textbf{2023/50/E/ST1/00025.} \] For the purpose of Open Access, the author has applied a CC-BY public copyright license to any Author Accepted Manuscript (AAM) version arising from this submission.

\section*{Data Availability Statement}
No datasets were generated or analyzed during the current study.  The \textsc{Singular} code used for the computations is reproduced in full in \cref{app:code}.

\section*{Conflict of Interest}
The authors declare that they have no conflicts of interest related to this manuscript.

\appendix

\section{\textsc{Singular} verification}
\label{app:code}

The following script verifies all computational claims used in \cref{sec:proof}: the concurrent triples (\cref{lem:lattice}), the ranks in \cref{tab:ranks}, the syzygy \(\rho_f\), the total Tjurina numbers, and the minimal resolutions (\cref{prop:D0-resolutions}).  The ground field is \(\QQ\) (characteristic \(0\)), so all computations are exact.  It runs in less than one second with \textsc{Singular} 4.3.2.

{\small
\begin{verbatim}
// Verification of the nine-line counterexample.
// All computations are exact: the ground field is Q (characteristic 0).
LIB "matrix.lib";
ring S = 0,(x,y,z),dp;
list LA = y-z, y+2z, 2y+z, x-2y-z, x-y-2z, x-y+z, x+y-z, x+y+2z, 2x-y-2z;
list LB = x, x-y, x+y, x-y-z, y+z, z, x-z, x+y+2z, x-2y-z;
poly f = 1; poly g = 1; int i, q;
for (i = 1; i <= 9; i++) { f = f*LA[i]; g = g*LB[i]; }

// (1) Concurrent triples {i,j,k}: det of the coefficient vectors = 0.
proc concurrentTriples(list L)
{
  string s; matrix M[3][3]; int i, j, k, t; intvec v;
  for (i = 1; i <= 7; i++) { for (j = i+1; j <= 8; j++) {
    for (k = j+1; k <= 9; k++) {
      v = i, j, k;
      for (t = 1; t <= 3; t++)
      { M[t,1] = diff(L[v[t]],x); M[t,2] = diff(L[v[t]],y);
        M[t,3] = diff(L[v[t]],z); }
      if (det(M) == 0) { s = s + " {" + string(v) + "}"; }
  }}}
  return(s);
}
string TA = concurrentTriples(LA); string TB = concurrentTriples(LB);
"concurrent triples (A):" + TA;
"concurrent triples (B):" + TB;
"same list for A and B: " + string(TA == TB);

// (2) rank of Phi_{h,q} : S_q^3 -> S_{q+8}, (a,b,c) -> a h_x + b h_y + c h_z.
//     Columns = images of the monomial basis of S_q^3, written in the
//     monomial basis of S_{q+8}.
proc rankPhi(poly h, int q)
{
  ideal m = maxideal(q);                              // monomials of degree q
  ideal P = m*diff(h,x), m*diff(h,y), m*diff(h,z);    // 3*dim S_q polynomials
  matrix C = coeffs(P, maxideal(q+8));                // dim S_{q+8} x 3*dim S_q
  return(rank(C));
}
for (q = 0; q <= 7; q++)
{
  printf("q=%s: dim S_q^3=%s, rank Phi_f=%s, rank Phi_g=%s",
         q, 3*size(maxideal(q)), rankPhi(f,q), rankPhi(g,q));
}

// (3) the explicit Jacobian syzygy of degree 4 for f
poly A = -x*(6x3-17x2y-10x2z-14xy2-20xyz-47xz2+33y3-15y2z+75yz2+69z3);
poly B = 12x3y-28x2y2+10x2yz+18x2z2-4xy3+20xy2z-61xyz2-36xz3+12y4-57y3z
         +33y2z2+75yz3+18z4;
poly C = 12x3z+18x2y2-x2yz-17x2z2-40xy2z-34xyz2-7xz3-18y4+39y3z+69y2z2
         -3yz3-6z4;
"A*f_x + B*f_y + C*f_z = " + string(A*diff(f,x) + B*diff(f,y) + C*diff(f,z));

// (4) total Tjurina numbers = degrees of the Jacobian schemes
"tau(A) = " + string(mult(std(jacob(f))));
"tau(B) = " + string(mult(std(jacob(g))));

// (5) minimal graded free resolutions of the Milnor algebras M(f), M(g)
resolution rf = mres(jacob(f), 0); print(betti(rf), "betti");
resolution rg = mres(jacob(g), 0); print(betti(rg), "betti");
\end{verbatim}
}
\end{document}